\documentclass[reqno]{amsart}

\usepackage{amsmath,mathrsfs,amssymb,multirow,setspace,enumerate}
\usepackage{graphicx}
\usepackage{tikz}
\usepackage[colorlinks=true,allcolors=black]{hyperref}
\usepackage{relsize}

\theoremstyle{plain}
\usepackage[a4paper, total={7in, 9in}]{geometry}
\newtheorem{theorem}{Theorem}[section]
\newtheorem{lemma}[theorem]{Lemma}
\newtheorem{proposition}[theorem]{Proposition}
\newtheorem{corollary}[theorem]{Corollary}

\theoremstyle{definition}

\newtheorem{example}[theorem]{Example}
\newtheorem{definition}[theorem]{Definition}
\newtheorem{claim}[theorem]{Claim}
\theoremstyle{remark}
\newtheorem{remark}[theorem]{Remark}

\input xy
\xyoption{all}
\begin{document}
	\title[On The Enhanced Power Graph of a Semigroup]{On The Enhanced Power Graph of a Semigroup}
	\author[Sandeep Dalal, Jitender Kumar, Siddharth Singh]{Sandeep Dalal, Jitender Kumar, Siddharth Singh}
	\address{Department of Mathematics, Birla Institute of Technology and Science Pilani, Pilani, India}
	\email{deepdalal10@gmail.com,jitenderarora09@gmail.com,sidharth$\_$0903@hotmail.com}

	\begin{abstract}
		The enhanced power graph $\mathcal P_e(S)$ of a semigroup $S$ is a simple graph whose vertex set is $S$ and two vertices $x,y \in S$ are adjacent if and only if $x, y \in \langle z \rangle$ for some $z \in S$, where $\langle z \rangle$ is the subsemigroup generated by $z$. In this paper, first we described the structure of $\mathcal P_e(S)$ for an arbitrary semigroup $S$. Consequently, we discussed the connectedness of  $\mathcal P_e(S)$. Further, we characterized the semigroup $S$ such that $\mathcal P_e(S)$ is complete, bipartite, regular, tree and null graph, respectively. Also, we have investigated the planarity together with the minimum degree and independence number of  $\mathcal P_e(S)$. The chromatic number of a spanning  subgraph, viz. the cyclic graph, of  $\mathcal P_e(S)$ is proved to be countable. At the final part of this paper, we construct an example of a semigroup $S$ such that the chromatic number of  $\mathcal P_e(S)$ need not be countable. 
	\end{abstract}

	
	
	\maketitle

\section{Introduction}
The investigation of graphs associated to semigroups is a large research area. In 1964, Bosak \cite{b.bosak1964graphs} studied certain graphs over semigroups. Most important class of graphs defined by semigroups is that Cayley graphs (cf. \cite{a.budden1985cayley,a.trotter1978cartesian,a.witte1984survey}), since they have numerous application (cf. \cite{a.kelarev2009mining,a.Kelerve-minimal-automata}). In 2000, Kelarev and Quinn defined two interesting classes of directed graphs, viz.  divisibility and power graph on semigroups \cite{a.kelarev2002directed,a.kelarev2001powermatrices}). The undirected power graph $\mathcal{P}(S)$ of a semigroup $S$  became the main focus of study in \cite{a.MKsen2009} defined by Chakraborty et al. is whose vertex set is $S$ and two distinct vertices $x, y$ are adjacent if either $x = y^m$ or $y = x^n$  for some $m, n \in \mathbb N$. We refer the reader for more results on power graph to survey paper \cite{a.abawajy2013power}.  The commuting graph $\Delta(S)$ of a semigroup $S$ is the graph whose vertex set is $\Omega \subseteq S$ and two distinct vertices $x, y$ are adjacent if $xy = yx$. Commuting graphs of various semigroups have been studied in term of their properties such as connectivity or diameter (cf.\cite{a.Araujo2015,a.Araujo2011,a.konieczny2002semigroups}). The power graph $\mathcal P(S)$ is a spanning subgraph of commuting graph $\Delta (S)$ when $\Omega = S$. In \cite{a.Aalipour2017}, Aalipour et al. characterized the finite group $G$ such that the $\mathcal P(G)$ coincides with the commuting graph of $G$. If these two graphs of $G$ do not coincide, then to measure how much the power graph is close to the commuting graph of a group $G$, they introduced a new graph called enhanced power graph, denoted by $\mathcal{P}_e(G)$, is the graph whose vertex set is the group $G$ and two distinct vertices $x, y$ are adjacent if $x, y \in \langle z \rangle$ for some $z \in G$. Aalipour et al. characterize the finite group $G$ in \cite{a.Aalipour2017}, for which equality holds for either two of the three graphs viz. power graph, enhanced power graph and commuting graph of $G$. Further, the enhanced power graphs have been studied by various researchers. In \cite{a.Bera2017}, Bera et al. characterized the abelian groups and the non abelian $p$-groups having dominatable enhanced power graphs. In \cite{a.Dupont2017}, Dupont et al. determined the rainbow connection number of enhanced power graph of a finite group $G$. Later, Dupont et al.   studied the graph theoretic properties in \cite{a.Dupont2017quotient} of enhanced quotient graph of a finite group $G$.  Ma et al. \cite{2019Mametric} investigated the metric dimension of an enhanced power graph of finite groups.  Zahirovi$\acute{c}$ et al. \cite{a.2019study}  proved that two finite abelian groups are isomorphic if their enhanced power graphs are isomorphic. Also, they supplied a characterization of finite nilpotent groups whose enhanced power graphs are perfect.  Recently, Panda et al. \cite{a.Panda-enhanced} studied the graph-theoretic properties viz. minimum degree, independence number, matching number, strong metric dimension and perfectness of enhanced power graph over finite abelian groups and some non abelian groups such as Dihedral groups, Dicyclic groups and the group $U_{6n}$.  Dalal et al. \cite{a.dalal2021enhanced} investigated the graph-theoretic properties of enhanced power graphs over semidihedral group $SD_{8n}$ and the group $V_{8n}$.  Bera et al. \cite{a.Bera2021EPG} gave an upper bound for the vertex connectivity of enhanced power graph of any finite abelian group $G$. Moreover, they classified the finite abelian group $G$ such that their proper enhanced power graph is connected.

In this paper, we have initiated the study of enhanced power graph on a semigroup $S$.  This paper is structured as follows. In Section 2, we provide necessary background material and fix our notations used throughout the paper. In Section $3$, first we describe the structure of $\mathcal{P}_e(S)$ and then study various graph theoretic properties viz. connectedness, completeness, bipartite, minimum degree, independence number etc. Additionally, we have studied the planarity of $\mathcal{P}_e(S)$. In Section $4$, we provide an example of a semigroup $S$ such that $\chi(\mathcal{P}_e(S))$ is uncountable.

\section{Preliminaries}
In this section, we recall necessary definitions, results and notations of graph theory from \cite{b.West} and semigroup theory from \cite{b.Howie}. 
A graph $\mathcal{G}$ is a pair  $ \mathcal{G} = (V, E)$, where $V = V(\mathcal{G})$ and $E = E(\mathcal{G})$ are the set of vertices and edges of $\mathcal{G}$, respectively. We say that two different vertices $a, b$ are $\mathit{adjacent}$, denoted by $a \sim b$, if there is an edge between $a$ and $b$. We are considering simple graphs, i.e. undirected graphs with no loops  or repeated edges. If $a$ and $b$ are not adjacent, then we write $a \nsim b$. The \emph{neighbourhood} $ N(x) $ of a vertex $x$ is the set all vertices adjacent to $x$ in $ \mathcal G $. Additionally, we denote ${\rm N}[x] = {\rm N}(x) \cup \{x\}$. A subgraph  of a graph $\mathcal{G}$ is a graph $\mathcal{G}'$ such that $V(\mathcal{G}') \subseteq V(\mathcal{G})$ and $E(\mathcal{G}') \subseteq E(\mathcal{G})$. A \emph{walk} $\lambda$ in $\mathcal{G}$ from the vertex $u$ to the vertex $w$ is a sequence of  vertices $u = v_1, v_2,\ldots, v_{m} = w$ $(m > 1)$ such that $v_i \sim v_{i + 1}$ for every $i \in \{1, 2, \ldots, m-1\}$. If no edge is repeated in $\lambda$, then it is called a \emph{trail} in $\mathcal{G}$. A trail whose initial and end vertices are identical is called a \emph{closed trail}. A walk is said to be a \emph{path} if no vertex is repeated.  The length of a path is the number of edges it contains and a path of length $n$ is denoted by $P_n$. If $U \subseteq V(\mathcal{G})$, then the  subgraph of $\mathcal{G}$ induced by  $U$ is the graph $\mathcal{G}'$ with vertex set $U$, and with two vertices adjacent in $\mathcal{G}'$ if and only if they are adjacent in $\mathcal{G}$. A graph  $\mathcal{G}$ is said to be \emph{connected} if there is a path between every pair of vertex. A graph $\mathcal{G}$ is said to be \emph{complete} if any two distinct vertices are adjacent.  A path that begins and ends on the same vertex is called a \emph{cycle}.

An \emph{independent set} of a graph $\mathcal{G}$ is a subset  of $V(\mathcal{G})$ such that no two vertices in the subset are adjacent in $\mathcal{G}$. 
A graph $\mathcal{G}$ is said to be \emph{bipartite} if $V(\mathcal{G})$ is the union of two disjoint independent sets.  A graph $\mathcal{G}$ is called a \emph{complete bipartite} if $\mathcal{G}$ is bipartite with $V(\mathcal{G}) = A \cup B$, where $A$ and $B$ are disjoint independent sets such that $x \sim y$ if and only if $x \in A$ and $y \in B$. We shall denote it by $K_{n, m}$, where $|A| = n$ and $|B| = m$. A graph $\mathcal{G}$ is said to be a \emph{star graph}\index{star graph} if $\mathcal{G} = K_{1, n}$ for some $n \in \mathbb N$. 

\begin{theorem}[{\cite[Theorem 1.2.18]{b.West}}]\label{ch1-bipartite}
A graph  $\mathcal{G}$ is bipartite if and only if $\mathcal{G}$ does not contain an odd cycle.
\end{theorem}
%
%
%
%
%
%

A \emph{planar graph}\index{planar} is a graph that can be embedded in the plane, i.e. it can be drawn on the plane in such a way that its edges intersect only at their endpoints. The following theorem will be useful in the sequel.

\begin{theorem} [{\cite[Theorem 6.2.2]{b.West}}] \label{ch1-planar-Kuratowski}
A graph $\mathcal{G}$ is planar if and only if it does not contain a subdivision of $K_5$ or $K_{3,3}$.
\end{theorem}



Now we anamnesis some basic definitions and results on semigroups. A \emph{semigroup} is a non-empty set $S$ together with an associative binary operation on $S$. We say $S$ to be a \emph{monoid}\index{Monoid} if it contains an identity element $e$. If $S$ has no identity element, then it is easy to adjoin an extra element $e$ to $S$ to form a monoid. We define $es = se = s ~ \forall s \in S$ and $ee = e$, it is routine to verify that $S \cup \{e\}$ forms a monoid and it is denoted by $S^1$.  A monoid $S$ is said to be a \emph{group} if for each $x$ there exists $x^{-1} \in S$ such that $xx^{-1} = x^{-1}x = e$. A \emph{subsemigroup} of a semigroup is a subset that is also a semigroup under the same operation. A subsemigroup of $S$ which is a group with respect to the multiplication inherited from $S$ will be called  \emph{subgroup}. An element $a$ of a semigroup $S$ is \emph{idempotent}\index{idempotent element}   if $a^2 = a$ and the set of all idempotents in $S$ is denoted by $E(S)$.  A \emph{band}\index{band} is a semigroup  in which every element is idempotent. For a subset $X$ of a semigroup $S$, the subsemigroup generated by $X$, denoted by $\langle X \rangle$, is the intersection of all the subsemigroup of $S$ containing $X$ and it is the smallest subsemigroup of $S$ containing $X$. The subsemigroup $\langle X \rangle$ is the set of all the elements in $S$ that can be written as finite product of elements of $X$. If $X$ is finite then $\langle X \rangle$ is called finitely generated subsemigroup of $S$. A semigroup $S$ is called \emph{monogenic}\index{monogenic semigroup} if there exists $a \in S$ such that $S = \langle a \rangle$. Clearly, $\langle a \rangle = \{a^m \; : \; m \in \mathbb{N}\}$, where $\mathbb{N}$ is the set of positive integers. A subgroup generated by $X$ can be defined analogously. If $X = \{a\}$, then the subgroup generated by $X$ is called \emph{cyclic}\index{cyclic subgroup}. Note that the cyclic subgroup generated by $a$ is $\langle a \rangle = \{a^m : \; m \in \mathbb Z \}$.

For  $X \subseteq S$, the number of elements in $X$ is called the order (or size) of $X$ and it is denoted by $|X|$. The $\mathit{order}$ of an element $a\in S$, denoted by $o(a)$, is defined as  $|\langle a \rangle|$. The set $\pi(S)$ consists order of all the elements of a semigroup $S$. In case of finite monogenic semigroup, there are repetitions among the powers of $a$. Then the set
\[\{x \in \mathbb{N} : (\exists \; y \in \mathbb{N}) a^x = a^y, x \ne y\}\]
is non-empty and so has a least element. Let us denote this least element by $m$ and call it the \emph{index}\index{index} of the element $a$. Then the set
\[\{x \in \mathbb{N} \; : \; a^{m + x} = a^m \}\]
is non-empty and so it too has a least element $r$, which we call the \emph{period}\index{period} of $a$. Let $a$ be an element  with index $m$ and period $r$. Thus, $a^m = a^{m + r}$. It follows that $a^m = a^{m + qr}$ for all $q \in \mathbb{N}$. By the minimality of $m$ and $r$ we may deduce that the powers
$a, a^2, \ldots, a^m, a^{m + 1}, \ldots, a^{m + r-1}$
are all distinct. For every $s \ge m$, by division algorithm we can write $s = m + qr + u$, where $q \ge 0$ and $0 \le u \le r-1$. then it follows that

\[a^s = a^{m + qr}a^u = a^m a^u = a^{m+u}.\]
Thus, $\langle a \rangle = \{a, a^2, \ldots, a^{m + r-1}\}$ and $o(a) = m + r - 1$. The subset \[\mathcal{K}_a = \{a^m, a^{m+1}, \ldots, a^{m+r-1} \}\]  is a subsemigroup of $\langle a \rangle$.  Moreover, there exists $g \in  \mathbb N$ such that $0 \leq g \leq r-1$ and $m + g \equiv 0({\rm mod} \; r)$. Note that $a^{m + g}$ is the idempotent element and so it is the identity element of $\mathcal{K}_a$. Because $a^{(m + g)^2} = a^{2m + 2g} = a^{m + (m +g) + g} = a^{m + tr + g}$ as $m + g \equiv 0({\rm mod} \; r)$ which gives $a^{(m + g)^2} = a^{m +g}$. If we choose $ g' \in  \mathbb N$ such that 
\[0 \leq  g' \leq r-1 \; {\rm and} \; m + g' \equiv 1({\rm mod} \; r),\]

\noindent then $k(m + g') \equiv k ({\rm mod} \; r)$ for all $k \in \mathbb N$, and so the powers $(a^{m + g'})^k$ of $a^{m +g'}$ for $k = 1, 2, \ldots, r$, deplete $\mathcal{K}_a$. Thus, $\mathcal{K}_a$ is the cyclic subgroup of order $r$, generated by $a^{m +g'}$.
Let $a$ be an element of a semigroup $S$ with index $m$ and period $r$. Then the monogenic semigroup $\langle a \rangle$ is denoted by $M(m, r)$. Also, sometimes $M(m, r)$ shall be written as
$\langle a : a^m = a^{m + r}\rangle$. The notations $m_a$ and $r_a$ denotes the index and period of $a$ in $S$, respectively. It is easy to observe that index of every element in a finite group $G$ is one. Consequently, for $a \in G$, we have $\langle a \rangle$ is the cyclic subgroup of $G$. 
\begin{remark}\label{ch1-Ka^i}
	Let $S = M(m, r) = \langle a \rangle$ be a monogenic semigroup. Then $\mathcal{K}_{a^i} = \langle a^i \rangle \cap \mathcal{K}_a$.
\end{remark}

A \emph{maximal monogenic subsemigroup} of $S$ is a monogenic subsemigroup of $S$ that is not properly contained in any other monogenic subsemigroup of $S$. We shall denote $\mathcal{M}$  by the set of all elements of $S$ that generates maximal monogenic subsemigroup of $S$ i.e.
\[\mathcal{M} = \{a \in S : \; \langle a \rangle \; {\rm is \; a \; maximal \; monogenic \; subsemigroup \; of} \; S \}.\] 

\emph{Green's relations}\index{Green's relations} were introduced by J.A  Green in $1951$ that characterize the elements of $S$  in terms of principal ideals. Further, they become a standard tool for investigating the structure of semigroup. These relations are defined as follows.
\begin{enumerate}
	\item $ x\; \mathcal L \; y$ if and only if $S^1 x = S^1 y$.
	\item $x \; \mathcal R  \; y$ if and only if $x S^1  = y S^1 $.
	\item $x \; \mathcal J \;  y$ if and only if $S^1 x S^1 = S^1 y S^1$.
	\item $x \; \mathcal H \; y$ if and only if $x \; \mathcal L  \; y$ and $x \; \mathcal R \; y$.
	\item $x \; \mathcal D \; y$ if and only if $x \; \mathcal L \; z$ and $z \; \mathcal R \; y$ for some $z \in S$.
\end{enumerate}
\vspace{.2cm}

\begin{remark}[{{\cite[p. $46$]{b.Howie}}}]\label{ch1-re-Green's-class}
	Let $G$ be a group. Then $\mathcal L = \mathcal R = \mathcal H = \mathcal D = \mathcal J = G \times G.$
\end{remark}

\begin{corollary}[{{\cite[Corollary 2.2.6]{b.Howie}}}]\label{ch1-H-class}
	Let $S$ be a semigroup and $f$ be an idempotent element of $S$. Then the $\mathcal{H}$-class $H_f$ containing $f$ is a  subgroup of $S$. 
\end{corollary}

A semigroup is said to be \emph{completely regular}\index{completely regular semigroup} if every element $a$ of $S$ lies in a subgroup of $S$. Further, we have the following characterization of completely regular semigroup.

\begin{proposition}[{{\cite[Proposition 4.1.1]{b.Howie}}}]\label{ch1-completely-regular}
A semigroup $S$ is completely regular if and only if every $\mathcal{H}$-class in $S$ is a group.
\end{proposition}

A semigroup $S$ is said to be of \emph{bounded exponent}\index{bounded exponent} if there exists a positive integer $n$ such that for all $x \in S$, $x^n = f$ for some $f \in E(S)$. If $S$ is of bounded exponent then the \emph{exponent}\index{exponent} of $S$ is the least $n$ such that for each $x \in S$,  $x^n = f$ for some $f \in E(S)$. Note that every finite semigroup is of bounded exponent. We often use the following fundamental properties of semigroups without referring to it  explicitly. Let $S$ be a semigroup of bounded exponent. For $f \in E(S)$, we define
\begin{align}\label{ch1-eq-1}
	S_f = \{ a \in S \; : \; a^m= f \; \text{for some} \; m \in \mathbb N\}.
\end{align}

The following remark is useful in the sequel.

\begin{remark}\label{ch1-re-classification-compoenent} Let $S$ be a semigroup of bounded exponent. Then $S = \underset{f \in E(S)}{\bigcup S_f}$ and for distinct $f, \;  f' \in E(S)$, we have $S_f \bigcap S_{f'} = \emptyset$.
\end{remark}

\section{Graph invariants of $\mathcal P_e(S)$}
		
In this section, we first describes the structure of the enhanced power graph of a semigroup $S$. Further,  we characterized the semigroup $S$ such that $\mathcal{P}_e(S)$ is complete, connected, bipartite, tree and regular, respectively. Finally, we characterize the semigroup $S$ such that $\mathcal P_e(S)$ is planar.  For $x \in S$ and $m,n \in \mathbb N$, we define 
\[S(x,m,n) = \{y \in S : \; x^m = y^n \}\]
and we write  $C(x) = \bigcup\limits_{m,n \in  \mathbb N}S(x,m,n)$. The
following proposition describes the structure of $\mathcal P_e(S)$.
\begin{proposition}\label{ch2-component-infinite}
	The set $C(x)$ is a connected component of $\mathcal P_e(S)$. Moreover, the components of the graph $\mathcal P_e(S)$ are precisely $\{C(x) \; | \; x \in S \}$.
\end{proposition}
\begin{proof}
	Let $y, z \in C(x)$. Then $y \in S(x,m,n)$ and $z \in S(x,p,q)$.  It follows that $x^m = y^n$ and $x^p = z^q$. Note that $y \sim x^{mp} \sim z$. Thus, $C(x)$ is connected in $\mathcal P_e(S)$. Now suppose that the element $z$ of $S$ is adjacent to a vertex $y$ in $C(x)$. Since $y \sim z$ implies $ y, z \in  \langle t \rangle$ for some $t \in S$. For $y \in  C(x)$, we have $x^m = y^n$ for some $m, n \in \mathbb N$.  It follows that $y= t^{\alpha}, \; z = t^{\beta}$ and  $z^{n\alpha} = x^{m\beta}$. Thus, $z \in  C(x)$. Hence,  $C(x)$ is a connected component of $\mathcal P_e(S)$.
\end{proof}

\begin{corollary}\label{connected} 
Let $S$ be a semigroup. Then $\mathcal{P}_e(S)$ is connected if and only if $\langle x \rangle \cap \langle y \rangle \neq \varnothing$ for all $x, y \in S$. In this case, {\rm diam}$(\mathcal{P}_e(S)) \leq 2$. 
\end{corollary}

The following lemma is useful in the sequel.

\begin{lemma}\label{gena-one idempotent}
Let $a$ be an element of a finite semigroup $S$. Then the subsemigroup $\langle a \rangle$ contains exactly one idempotent.
\end{lemma}

\begin{proof}
Let $m$ and $r$ be the index and period of $a$, respectively. Then, $\langle a \rangle = \{a, a^2, \ldots, a^{m + r-1}\}$. The subgroup $\mathcal{K}_a = \{a^m, a^{m+1}, \ldots, a^{m+r-1} \}$ of $\langle a \rangle$ contains exactly one idempotent. If for $1 \le i < m$, $a^i$ is an idempotent, then we have $a^{2i} = a^i$. Consequently, $m \le i$; a contradiction. Hence, the idempotent element of $\mathcal{K}_a$ is the only idempotent in $\langle a \rangle$.
\end{proof}

\begin{theorem}\label{connected-component}
	Let $S$ be a semigroup of bounded exponent. Then $S_f$ is a connected component of $\mathcal{P}_e(S)$ with unique idempotent $f$. Moreover, the connected components of $\mathcal{P}_e(S)$ are precisely $\{S_f \; : \; f \in E(S)\}$ and the number of connected components of $\mathcal{P}_e(S)$ is equal to $|E(S)|$.
\end{theorem}

\begin{proof}
	Suppose $a, b \in S_f$. If any one of them is $f$, then $a \sim b$ in $\mathcal{P}_e(S)$. If $a, b \in S_f \backslash \{f\}$, then there is a path $a \sim f \sim b$ in $\mathcal{P}_e(S)$. Let if possible $x \in S \backslash S_f$ such that $x \sim b$ for some $b \in S_f$ so that $x, b \in \langle c \rangle$ for some $c \in S$. Since $b^n = f$ for some $n \in \mathbb N$ so that $f \in \langle c \rangle$. Consequently we get $f \in  \langle y \rangle$ for all $y \in \langle c \rangle$ (see Lemma \ref{gena-one idempotent}). It follows that $x \in S_f$; a contradiction. Hence, $S_f$ is a connected component of $\mathcal{P}_e(S)$ and  by Remark \ref{ch1-re-classification-compoenent}, the connected components of $\mathcal{P}_e(S)$ are $\{S_f\; : \; f \in E(S)\}$.
\end{proof}
\begin{corollary}
A semigroup $S$ is a band if and only if $\mathcal{P}_e(S)$ is a null graph. 
\end{corollary}

Unless stated otherwise, hereafter $S_f$ always denotes the connected component of $\mathcal{P}_e(S)$ containing the idempotent $f$. The following lemma is useful in the sequel.

\begin{lemma}\label{ch1-exponent}
	Let $S$ be a semigroup with exponent $n$. Then for $x \in S$, we have
	\begin{enumerate}
		\item[\rm (i)] $o(x) \leq 2n$ for all $x \in S$.
		
		\item[\rm (ii)] the subsemigroup $\langle x \rangle$ is contained in some maximal monogenic subsemigroup of $S$.
	\end{enumerate}
\end{lemma}

\begin{proof}
	(i) Since $x^n = f$ for some $f \in E(S)$, we have $\langle x \rangle = M(m, r)$ for some $m, r \in \mathbb N$. There exists $g$ with $0 \leq g < r$ and $m +g \equiv 0 ({\rm mod} \; r)$ such that $x^{m + g} = f$. Clearly, $n \geq m$ as $x^n$ is the idempotent element of $\langle x \rangle$. Let if possible $n < m +g$. For $m \leq i \ne j \leq m + r- 1$, we have $x^i \ne x^j$ in $\langle x \rangle$. Therefore, $x^n \ne x^{m + g}$, which is not true because $x^n = x^{m + g} = f$. Thus, $n \geq m + g$. Also $r \mid m + g$ and $m \leq n$ gives $r \leq n$. It follows that $o(x) \leq m + r \leq 2n$. 
	
	\noindent (ii) If $\langle x \rangle$ is a maximal monogenic subsemigroup, then the result holds. Otherwise, $\langle x \rangle \subsetneq \langle x_1 \rangle$. If $\langle x_1 \rangle$ is maximal monogenic then this completes our proof. By (i), since $o(x) \leq 2n$, we get a finite chain such that $\langle x \rangle \subsetneq \langle x_1 \rangle \cdots \subsetneq \langle x_{k-1} \rangle \subsetneq \langle x_k \rangle$, where $k \leq 2n$ and $\langle x_k \rangle$ is a maximal monogenic subsemigroup of $S$. This complete the proof.
\end{proof}

\begin{theorem}\label{ch3-complete}
	Let $S$ be a semigroup with exponent $n$. Then $\mathcal{P}_e(S)$  is complete if and only if  $S$ is a monogenic semigroup.
\end{theorem}

\begin{proof}
	Let $S$ be a   monogenic semigroup. Then there exists $a \in S$ such that $S = \langle a \rangle$. For any  $x, y \in \mathcal{P}_e(S)$, we have $x, y \in   \langle a \rangle$. Thus, by definition, $\mathcal{P}_e(S)$ is complete. Conversely, suppose that  $\mathcal{P}_e(S)$ is complete. By Lemma \ref{ch1-exponent}, $o(x) \leq 2n$ for all $x \in S$.  Now choose an element $x \in S$ such that $o(x)$ is maximum. In order to prove that $S$ is monogenic, we show that  $S = H$, where $H = \langle x \rangle $. If $S \neq H$, then there exists $ y \in S$ but $y \notin H$. Since $\mathcal{P}_e(S)$ is complete,  $x, y \in \langle z \rangle$ for some $z \in S$.  Also note that $\langle z \rangle = \langle x \rangle$. Consequently, $y \in \langle x \rangle$; a contradiction. Thus, $S= \langle x \rangle$. Hence, $S$ is a monogenic semigroup.
\end{proof}

\begin{theorem}\label{bipartite} Let $S$ be a semigroup. Then the following statements are equivalent:
\begin{enumerate}
\item[\rm (i)] The set $\pi (S) \subseteq \{1,2\}$;
\item[\rm (ii)] $\mathcal{P}_e(S)$ is acyclic graph;
\item[\rm (iii)] $\mathcal{P}_e(S)$ is bipartite.
	\end{enumerate}	 
\end{theorem}

\begin{proof} 
	(i)$ \Rightarrow$(ii) Suppose $\pi(S) \subseteq \{1, 2\}$. Let if possible, there exists a cycle, viz. $a_0 \sim a_1 \sim \cdots \sim a_k \sim a_0$, in $\mathcal{P}_e(S)$. Since  this cycle must belongs to some connected component of $\mathcal{P}_e(S)$. As $o(a) \leq 2$ for all $a \in S$, we get $S$ is of bounded exponent.   By Theorem \ref{connected-component}, there exists $f \in E(S)$ such that $a_i \in S_f$ for all $i$, where $0 \leq i \leq k$. Consequently, at most one $i$ such that $a_i$ is an idempotent element. If none of the vertices of this cycle are idempotents, then $a_0, a_1, f \in \langle z_1 \rangle$ for some $z_1 \in S$ and $f \in E(S)$. Consequently, o$(z_1) \geq 3$; a contradiction. If one of the  vertex of the cycle $a_0 \sim a_1 \sim \cdots \sim a_k \sim a_0$ is idempotent, then note that there exist two non idempotent elements $a_i, a_j$ such that $a_i \sim a_j$. Thus, $a_i, a_j, f \in \langle z \rangle$ for some $z \in S$ which is not possible as $o(z) \leq 2$.
	
	\noindent (ii)$ \Rightarrow$(iii) Since $\mathcal{P}_e(S)$ is acyclic graph so that it does not contain any cycle. By Theorem \ref{ch1-bipartite}, $\mathcal{P}_e(S)$ is bipartite.
	
	\noindent (iii)$ \Rightarrow$(i) Suppose $\mathcal{P}_e(S)$ is a bipartite graph. By Theorem \ref{ch1-bipartite}, $\mathcal{P}_e(S)$ does not contain any odd cycle. To prove $\pi(S) \subseteq \{1, 2\}$. Let if possible there exists  $a \in S$ such that o$(a) \geq  3$, then $a \sim a^2 \sim a^3 \sim a$ is an odd cycle in $\mathcal{P}_e(S)$; a contradiction of the fact that $\mathcal{P}_e(S)$ is bipartite.
\end{proof}


\begin{corollary}
	Let $G$ be a group. Then the following statements are equivalent:
	\begin{enumerate}
		\item[\rm (i)]   exponent of $G$ is at most $2$. Moreover if $G$ is finite, then  $G \cong \mathbb Z_2 \times \mathbb Z_2 \times \cdots \times \mathbb Z_2$.
		\item[\rm (ii)] $\mathcal{P}_e(G)$ is acyclic graph;
		\item[\rm (iii)] $\mathcal{P}_e(G)$ is bipartite;
		\item[\rm (iv)] $\mathcal{P}_e(G)$ is a tree;
		\item[\rm (v)] $\mathcal{P}_e(G)$ is a star graph.
	\end{enumerate}	 
\end{corollary}

In view of Corollary \ref{connected} and Theorem \ref{bipartite}, we  have  the following corollary.

\begin{corollary}
	The enhanced power graph $\mathcal{P}_e(S)$ is a tree if and only if $|E(S)| = 1$ and $\pi (S) \subseteq \{1,2\}$.	
\end{corollary} 


\begin{theorem}\label{ch3-regular-graph} The enhanced power graph $\mathcal{P}_e(S)$ is $k$-regular\index{regular}  if and only if $S$ is the union of mutually disjoint monogenic subsemigroups of $S$ of size $k + 1$.
\end{theorem}

\begin{proof} Suppose $\mathcal{P}_e(S)$ is $k$-regular. Note that the order of each element is at most $k + 1$. Otherwise there exists $x \in S$ such that deg$(x) \ne k$; a contradiction of the fact that $\mathcal{P}_e(S)$ is $k$-regular. For each $x \in S$, there exist $m_x, g_x \in \mathbb N_0$ such that $m_x + g_x \leq k + 1$ and $a^{m_x + g_x} \leq k + 1$ is an idempotent element. Choose $n = (k + 1)!$ gives $a^n$ is an idempotent element for all $a \in S$.  Then $S$ is of bounded exponent. By Remark \ref{ch1-re-classification-compoenent},  $S = \underset{f \in E(S)}{\bigcup S_f}$. For $f, f' \in E(S)$, we get $|S_f| - 1 = {\rm deg}(f) = {\rm deg}(f') = |S_{f'}| - 1$. It follows that $|S_f| = |S_{f'}| = k + 1$. Since $\mathcal{P}_e(S)$ is regular, for each $a \in S_f$, we have deg$(a) = |S_f|-1$.  Consequently, for each  $f \in E(S)$, the subgraph induced by  $S_f$ is complete. Now we show that $S_f$ is a subsemigroup of $S$.
	Let $x, y \in S_f$. Then there exist $m, n \in \mathbb N$ such that $x^m = y^n = f$. Since $x \sim y$ as $S_f$ is complete so $x, y \in \langle z \rangle$ for some $z \in S$. Consequently, $\langle z \rangle \subseteq S_f$ so that $xy \in S_f$. By Theorem \ref{ch3-complete},  $S_f$ is a monogenic subsemigroup of $S$.
	
	Conversely, suppose $S$ is the union of mutually disjoint monogenic subsemigroup $S_i$ of $S$ of size $k + 1$ where $i \in \Lambda$ and $\Lambda$ is an index set. For our convenient, we assume that $S_i = \langle a_i \rangle$, where $i \in \Lambda$. Note that $a_i, a_j \notin S_f$ for some $f \in E(S)$. For instance if $a_i, a_j \in S_f$, then $f \in \langle a_i \rangle \cap \langle a_j \rangle$; a contradiction. Also, $f \in E(S) \subseteq S$ implies $f \in \langle a_i \rangle$ for some $i$. Consequently, $a_i \in S_f$ and so $\langle a_i \rangle \subseteq S_f$. If $x \in S_f \setminus \langle a_i \rangle$ then $x \in \langle a_j \rangle$ for some $j \ne i$. Therefore, $x \in S_{f'}$ for some $f' \ne f$; a contradiction of Remark  \ref{ch1-re-classification-compoenent}. Thus, for each $f \in E(S)$, we have $S_f = \langle a_i \rangle$ for some $i$. By Theorem \ref{ch3-complete}, for each $f \in E(S)$, the subgraph induced by $S_f$ is complete and by hypothesis the graph $\mathcal{P}_e(S)$ is regular. 
\end{proof}

By the similar lines of the  proof of Theorem \ref{ch3-regular-graph}, the following theorem on the completeness of the connected components of $\mathcal{P}_e(S)$.

\begin{theorem} Let $S$ be a semigroup of bounded exponent. Then the connected components of the enhanced power graph $\mathcal{P}_e(S)$ are complete\index{complete}  if and only if  $S$ is the union of mutually disjoint monogenic subsemigroup of $S$.
\end{theorem}

%
%

\begin{theorem}
	A semigroup $S$ with exponent $n$ is completely regular if and only if each connected component of $\mathcal{P}_e(S)$ forms a group.
\end{theorem}

\begin{proof}
	Suppose $S$ is completely regular semigroup. Then every $\mathcal H$-class of $S$ is a group (see Proposition \ref{ch1-completely-regular}). In view of Theorem \ref{connected-component}, each connected component of $\mathcal{P}_e(S)$ is of the form $S_f$ for some $f \in E(S)$. To prove that each connected component $\mathcal{P}_e(S)$ forms a group, we show that $S_f = H_f$ for each $f \in E(S)$. Let $a \in H_f$. Then $a^n = f$  for some $n \in \mathbb N$ as $S$ is of exponent $n$ so that $a \in S_f$. On the other hand, suppose $a \in S_f$. If $a \in H_{f'}$ for some $f' \ne f \in E(S)$, then $a \in S_{f'}$;  a contradiction. Thus $H_f = S_f$.
	
	Conversely, suppose every connected component of $\mathcal{P}_e(S)$ forms a group. To prove $S$ is completely regular, we show that every $\mathcal{H}$-class forms a group (see Proposition \ref{ch1-completely-regular}). Let $a \in S$. Then $a \in S_f$ for some $f \in E(S)$. We claim that $H_a = S_f$. Suppose $b \in S_f$. By Remark \ref{ch1-re-Green's-class}, $(b, f) \in \mathcal{H}$. Also, we have $(a, f) \in \mathcal{H}$ so that $(a, b) \in \mathcal{H}$. It follows that $S_f \subseteq H_a$. On the other hand let $b \in H_a$. Then $a \in S_f$ implies that $b \in H_f$. Since $H_f$ contains an idempotent so that $H_f$ forms a group (see Corollary \ref{ch1-H-class}). It follows that $b^m =  f$ for some $m \in \mathbb N$. Hence, $H_a  = S_f$ for some $f \in E(S)$.
\end{proof}





\begin{proposition}\label{ch3-isolated}
An element $a$ of an arbitrary semigroup $S$ is an isolated vertex in $\mathcal{P}_e(S)$ if and only if
\begin{enumerate}
\item[\rm (i)] $a$ is an idempotent in $S$.
\item[\rm (ii)] $ H_a = \{a\}$.
\item[\rm (iii)] $m_x = 1$ for each $x \in S_a$.
\end{enumerate}
\end{proposition}

\begin{proof}
	Let $a$ be an isolated vertex in $\mathcal{P}_e(S)$. Clearly $a \in E(S)$. Otherwise $a \sim a^2$. Consequently, $ H_a$ forms a group with the identity element $a$. Thus, every vertex of the enhanced power graph induced by $H_a$ will be adjacent with $a$. Let if possible, $x \in H_a \setminus \{a\}$ then $x \sim a$; a contradiction. If $b \ne a \in S_a$ then $b^m = a$ for some $ m \in \mathbb N$. It follows that $b \sim a$; a contradiction. Thus, $S_a = \{a\}$ and $m_a  = 1$.
	
	Conversely suppose for $a \in S$ satisfy (i), (ii) and (iii). Let if possible,  $a \sim x$ for some $x \in S$. Then $a, x \in \langle b \rangle$ for some $b \in S$. Consequently, $\langle x \rangle \subseteq \langle b \rangle$ and $a \in \langle x \rangle$. Note that $r_x  = 1$. If $r_x > 1$, then $a = x^p$ for some $p > m_x$. Consequently, $(a, x^q ) \in \mathcal H$ for some $m \leq q \ne p$ and  so  $| H_a| > 1$; a contradiction. Thus $r_x = 1$ and $m_x = 1$ implies  $x = a$; a contradiction.
\end{proof}

%

Now we discuss the planarity of $\mathcal{P}_e(S)$. We begin with the following proposition which ensures the non-planarity of $\mathcal{P}_e(S)$.

\begin{proposition}\label{ch3-planar-propo}
	Let $\mathcal{P}_e(S)$ be a planar graph\index{planar}. Then $o(a) < 5$ for all  $a \in S$.
\end{proposition}

\begin{proof}
	Let if possible, there exists an element $a \in S$ such that $o(a) \geq 5$. Then the subgraph induced by  $\langle a \rangle$ contains $K_5$. By Theorem \ref{ch1-planar-Kuratowski}, $\mathcal{P}_e(S)$  is non-planar; a contradiction.
\end{proof}

\begin{theorem}\label{ch3-planar}
	Let $S$ be a semigroup such that the index of every element of order four is either one or two. Then $\mathcal{P}_e(S)$ is planar  if and only if the following condition holds
	\begin{enumerate}
		\item[\rm (i)] For $a \in S$, we have o$(a) \leq 4$
		\item[\rm (ii)] $S$ does not contain $a, b, c \in S$ such that o$(a) = $o$(b) = $o$(c) = 4$, $m_a = m_b = m_c = 2$  and $|\langle a \rangle \cap \langle b \rangle \cap \langle c \rangle| = 3$.
	\end{enumerate}
\end{theorem}

\begin{proof}
	First suppose that $S$ satisfies the conditions (i) and (ii). In order to prove $\mathcal{P}_e(S)$ is planar, it is sufficient to show that every connected component of $\mathcal{P}_e(S)$ is planar.  Since $o(a) \leq 4$ for all $a \in S$ so that $S$ is of bounded exponent. In view of Theorem \ref{connected-component}, each connected component of $\mathcal{P}_e(S)$ is of the  form $S_f$, where $f \in E(S)$. Now we establish a planar drawing for $S_f$. Consider the set $A_4 = \{a \in S_f \; : \; \text{o}(a) = 4  \}$. In view of given hypothesis, note that the sets $A= \{a \in A_4 : m_a = 1 \}$ and $B = \{a \in A_4 : m_a = 2 \}$ forms a partition of $A_4$ i.e. $A_4 = A \cup B$. Let $a \in A_4$. Observe that $a \sim b$ if and only if $b \in \langle a \rangle$. For instance, if $a \sim b$ for some $b \in S \setminus \langle a \rangle$, then $a, b \in \langle c \rangle$  for some $c \in S$. Since $o(a) = 4$ and $o(c) \leq 4$, we have $b \in \langle a \rangle$; a contradiction.  Now  we prove that the subgraph induced by the elements of N$(A_4)$ is planar though the following claims:

	\begin{claim}
		The subgraph induced by  ${\rm N}[B]$ is planar.
	\end{claim} 
	
	\noindent\textit{Proof of claim:} First suppose that $a \in B$. Then  $m_a = 2$ so that $\langle a \rangle = \{a, a^2, a^3, a^4 : a^5 = a^2\}$ and $\langle a^2 \rangle = \langle a^4 \rangle$. It follows that  $a^2 \sim x$ if and only if $a^4 \sim x$ for all  $x \in S$. Further,  if $a^2 \sim x$ for some $x \in S \setminus \langle a \rangle$, then $a^2, x \in \langle y \rangle$ for some $y \in S$. By (i), $o(y) \leq 4$ and $x \in \{a^2, a^3, a^4\}$ gives $x = y$ and  $o(x) = 4$. By hypothesis, we get either $m_x = 2$ or $m_x = 1$. Since $o(a^2) = 3$ and $a^2 \in \langle x \rangle$ implies that $m_x = 2$.  If possible, let  $a^2 \sim y$ for some $y \in S \setminus (\langle a \rangle \cup \{x\})$. Then by using the similar argument, we get $\langle y \rangle  = \{y, a^2, a^3, a^4\}$ and $m_y = 2$. Therefore, we have $|\langle a \rangle \cap \langle x \rangle \cap \langle y \rangle| = 3$ and $m_a = m_x = m_y = 2$; a contradiction of (ii). It follows that either N$[\langle a \rangle] = \langle a \rangle$ or   N$[\langle a \rangle] = \langle a \rangle \cup \{x\}$ for some $x \in B$.  Note that  N$[x] = \langle x \rangle =  \{x, a^2, a^3, a^4\}$ as $x \in A_4$. Thus, we can draw the subgraph induced by all the vertices of N$(B)$ in a plane without cutting an edge shown in Figure \ref{fig-planar2}.
	\begin{figure}[h!]
		\centering
		\includegraphics[width=0.6\textwidth]{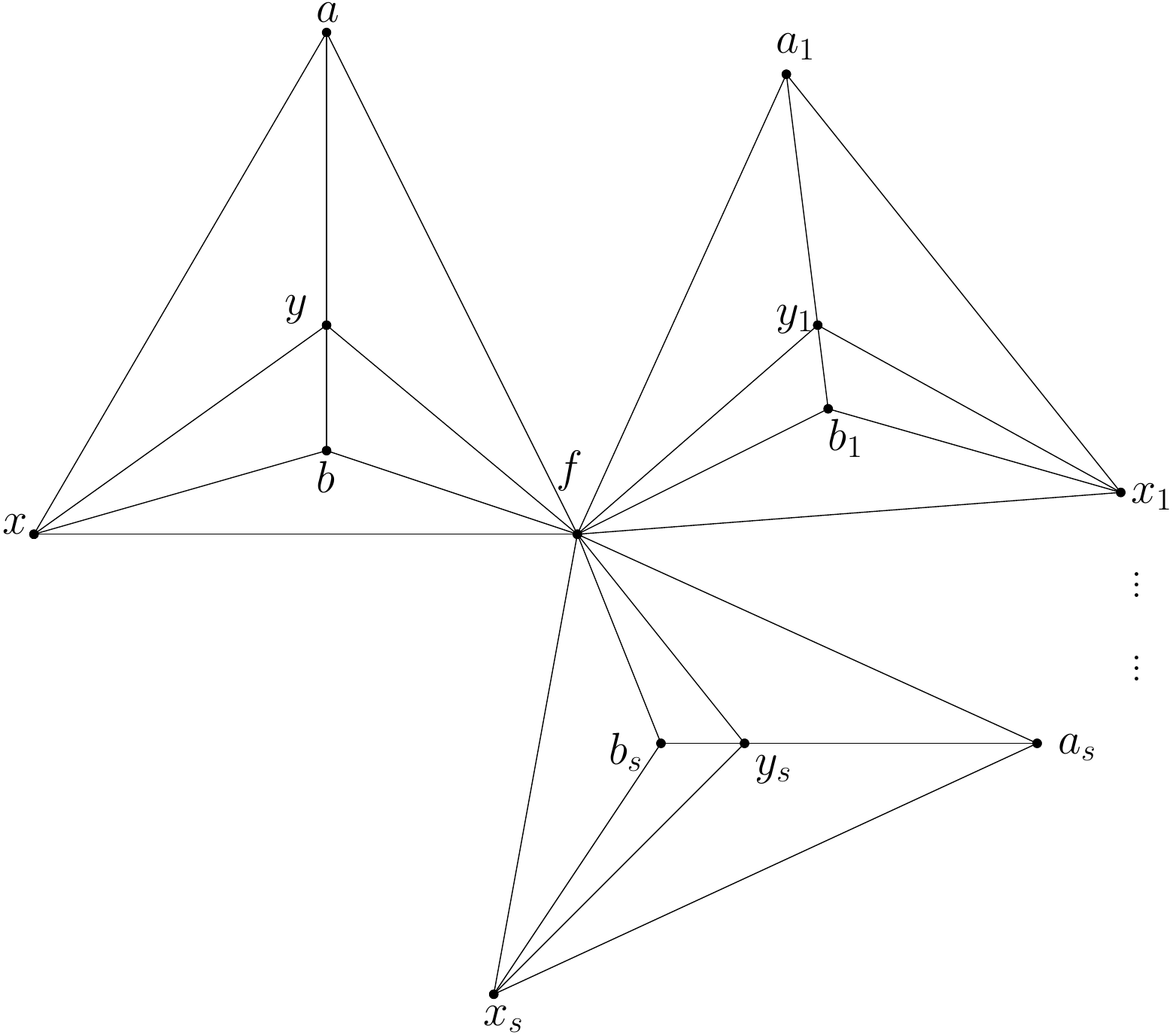}
		\caption{Planar drawing of $\mathcal{P}_e({\rm N}[B])$}\label{fig-planar2}
	\end{figure}

	\begin{claim}
		The subgraph induced by  ${\rm N}[A]$ is planar.
	\end{claim}
	
	\noindent\textit{Proof of claim:} Let $a \in A$. Then   $\langle a \rangle$ is a cyclic subgroup of order $4$ and N$[a] = {\rm N}[a^3] = \langle a \rangle$. If $a^2 \sim z$ for some $z \in S \setminus \langle a \rangle$, then we prove that N$[z] = \langle z \rangle$. Since  $a^2, z \in \langle t \rangle$ for some $t \in S$, we get $ 3 \leq o(t) \leq 4$ as $a^2, f, z$ are distinct element and $o(t) \leq 4$ (by (ii)).  If $o(t) = 3$,  then we must have $t = z$. Let if possible, $z \sim s$ for some $s \in S \setminus \langle z \rangle$, then $z, s \in \langle q \rangle$ for some $q \in S$. Since $o(z) = 3$ and $o(q) \leq 4$ (by (i)) implies $q = s$ and $o(s) = 4$. As a result, $m_s = 2$. Also $a^2 \in \langle s \rangle$ and $o(a^2) = 2$ which is not possible because $\langle s \rangle$ does not contain an element of order $2$. Therefore,  N$[z] = \langle z \rangle = \{f, z, a^2\}$. We may now assume that $o(t) = 4$. Note that $m_t \in \{1, 2\}$ as $t \in A_4$. For $o(a^2) = 2$ and $a^2, z \in \langle t \rangle$, we get $m_t = 1$ and $\langle t \rangle = \langle z \rangle$. It follows that  N$[z] = \langle z \rangle = \{f,z, a^2, z^3 = t\}$. Therefore, we have ${\rm N }[\langle a \rangle] = \langle a \rangle \cup P \cup Q$, where
	$P = \{ z\in S_f: \; {\rm N}[z] = \{z, a^2, f \} = \langle z \rangle \} $  and   $Q = \{ z\in S_f:\; {\rm N}[z] = \{z, a^2, z^3, f \} = \langle z \rangle  \}$. Thus
	the subgraph induced by the vertices of N$[{\rm N}[A]\setminus \{f\}]$ can be drawn in a plane without cutting an edge shown in Figure \ref{fig-planar1}. 
	
	Moreover, we observed that if $x \in {\rm N}[{\rm N}[A]\setminus \{f\}]$, then N$[x] \subseteq {\rm N}[{\rm N}[A]\setminus \{f\}]$.
	\begin{figure}[h!]
		\centering
		\includegraphics[width=0.8\textwidth]{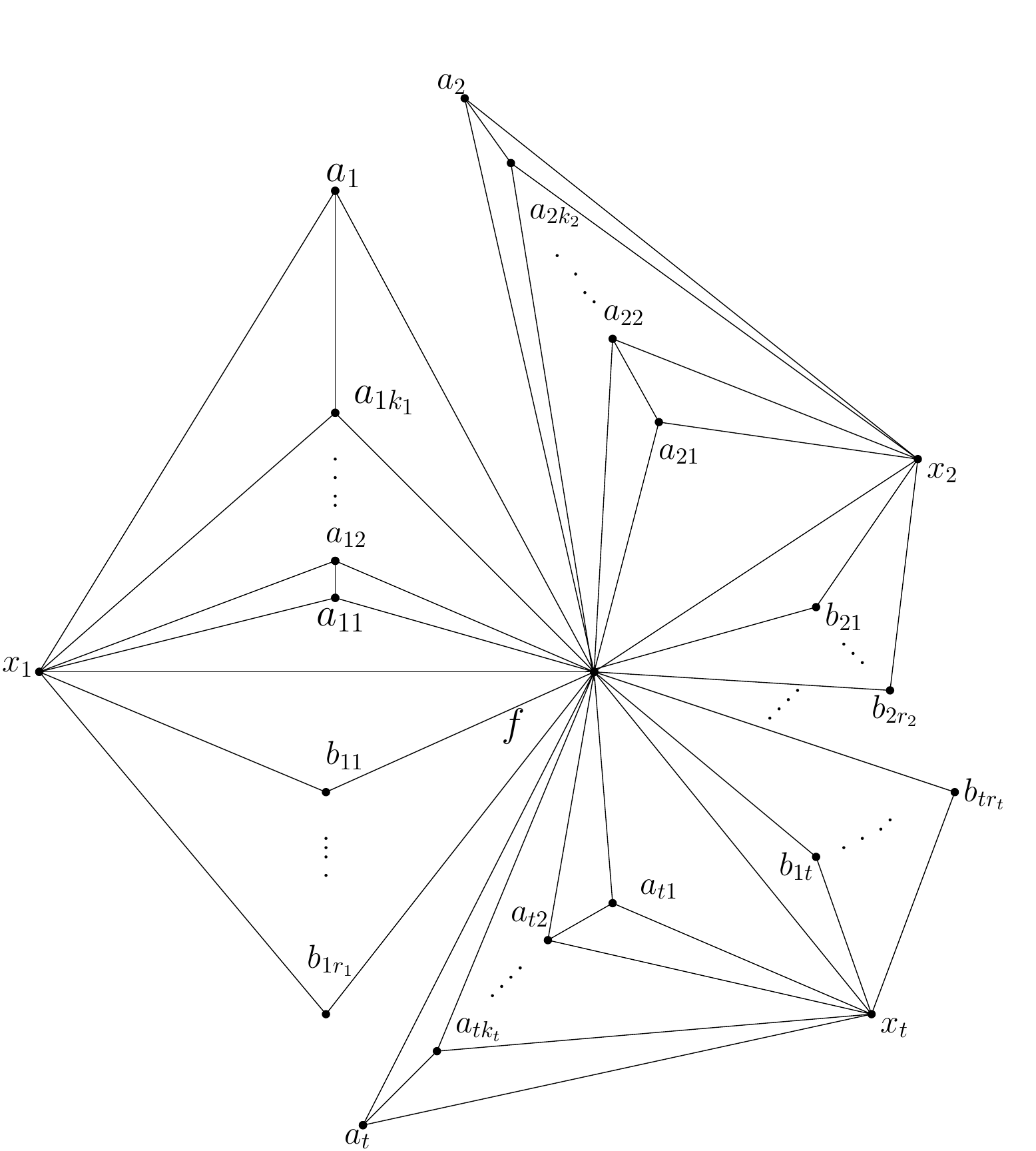}
		\caption{Planar drawing of $\mathcal{P}_e({\rm N}[A])$}\label{fig-planar1}
	\end{figure}

	\begin{claim}
		The subgraph induced by the neighbours of  $A_3 = \{x \in S_f \setminus {\rm N}(A_4)   \; : \; \text{o}(x) = 3\}$ is planar.
	\end{claim}
	
	\noindent\textit{Proof of claim:} Let $a \in A_3$ and $a \sim x$ for some $x \in S_f \setminus \langle a \rangle$. Note that $x \notin {\rm N}[A_4]$. Since $a, x \in \langle t \rangle$ for some $t \in S$ gives $a \in \langle x \rangle$ and $o(x) = 4$; a contradiction. If $m_a = 1$, then ${\rm N}[\langle a \rangle] = \langle a \rangle$. Otherwise, there exists $y_a \in \langle a \rangle$ such that $o(y_a) = 2$. Now we define $A_{y_a} = \{t \in A_3 \; : \; y_a \in \langle t \rangle \}$. For $t \in  A_{y_a}$, clearly, we have $t \in A_3$. Thus the subgraph induced by N$[A_3]$ can be drawn in a plane without cutting an edge. 
	
	Additionally, we can  conclude that the drawing of the subgraph induced by N$[A_3] \cup {\rm N}[{\rm N}[A]\setminus \{f\}]$ is planar. Now the set $A_2$ consists the remaining vertices is left in the subgraph induced by $S_f \setminus {\rm N}[A_3] \cup {\rm N}[{\rm N}[A]\setminus \{f\}]$. For $ x \in A_2$, we must have $o(x) = 2$ and N$(x) = \{f\}$. Thus, the result follows.
	
	Conversely, suppose $\mathcal{P}_e(S)$ is planar. Then by Proposition \ref{ch3-planar-propo}, $o(a) \leq 4$ for all $a \in S$.  On contrary, we assume $S$ does not satisfies the condition (ii). Then there exists $a, b, c \in S$ such that $o(a) = o(b) = o(c) = 4$, $m_a = m_b = m_c = 2$  and $|\langle a \rangle \cap \langle b \rangle \cap \langle c \rangle| = 3$. Therefore, we have $\langle a \rangle \cap \langle b \rangle \cap \langle c \rangle = \{x,y, z\}$. Consequently, $\mathcal{P}_e(S)$ contains a subgraph $K_{3,3}$ whose partitions  sets are $X = \{a, b, c\}$ and $Y = \{x, y, z\}$; a contradiction of the fact that $\mathcal{P}_e(S)$ is planar (see Theorem \ref{ch1-planar-Kuratowski}).
\end{proof}

In view of Theorem \ref{ch1-completely-regular},  we have the following corollaries of Theorem \ref{ch3-planar}.

\begin{corollary}
	Let $S$ be a completely regular semigroup. Then $\mathcal{P}_e(S)$ is planar \index{planar} if and only if $o(a) \leq 4$ for all $a \in S$.
\end{corollary}

\begin{corollary}{\rm \cite[Theorem 2.6]{a.Bera2017}}
	Let $G$ be a finite group. Then $\mathcal{P}_e(G)$ is planar if and only if $\pi(G) \subseteq \{1,  2, 3, 4\}$.
\end{corollary}


%

Now we construct  a semigroup $S$ which does not satisfy condition (ii) of Theorem \ref{ch3-planar} such that $\mathcal P_e(S)$ is non-planar. 

\begin{example}
	Consider the semigroup $S  = \{a, x, y, z, b, c  :   a^5 = a^2\}$ with the following Cayley table (see Figure \ref{tabel2}) and the enhanced power graph of a semigroup $S$ given in Figure \ref{fig-planar3}. Clearly, it contains a subgraph $K_{3, 3}$ with the partitioned sets $A = \{a,b,c\}$ and $B= \{x, y, z\}$.  So by Theorem \ref{ch1-planar-Kuratowski}, the  enhanced power graph of $S$ is non-planar. 
	
	\begin{figure}[!htb]
		\minipage{0.45\textwidth}
		\includegraphics[width=\linewidth]{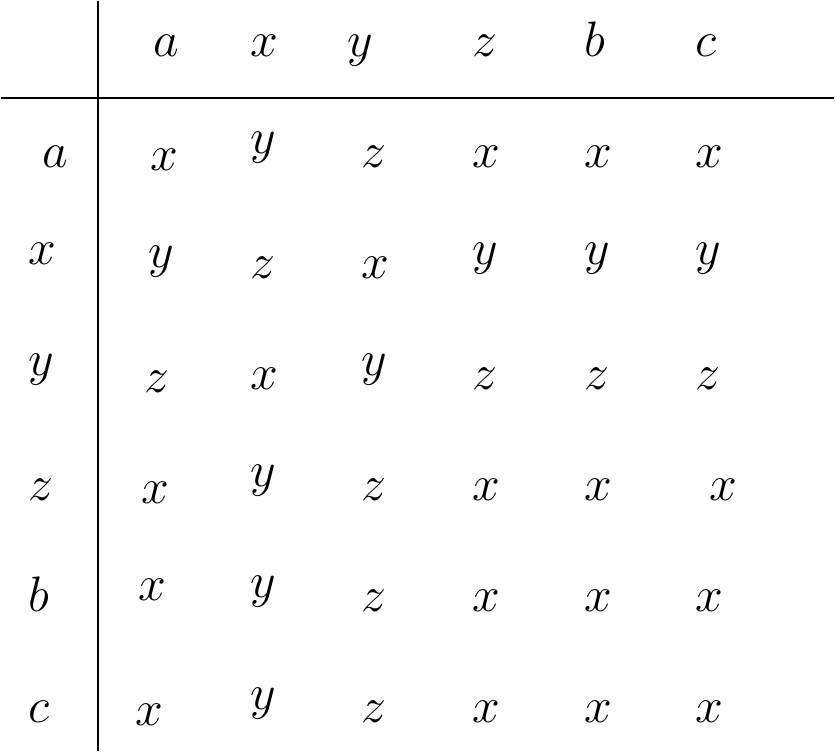}
		\caption{Cayley table}\label{tabel2}
		\endminipage\hfill
		\minipage{0.5\textwidth}
		\includegraphics[width=\linewidth]{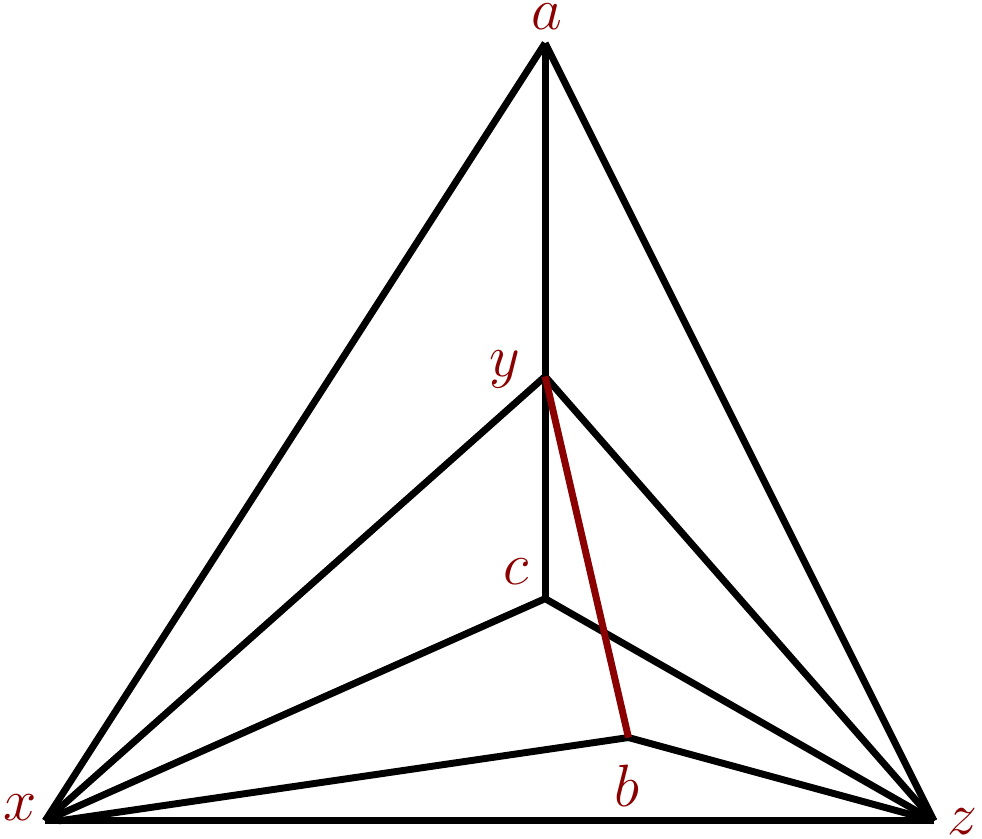}
		\caption{The enhanced power graph of $S$}\label{fig-planar3}
		\endminipage\hfill
		
	\end{figure}
	
\end{example}


Now we obtain the minimum degree and the independence number\index{independence number} of $\mathcal{P}_e(S)$.

\begin{theorem}
Let $S$ be a semigroup with exponent $n$. Then
\begin{enumerate}
\item[\rm (i)] $\delta(\mathcal{P}_e(S)) = m -1$, where $m = {\rm min}\{o(x) :   x \in \mathcal{M} \}$. 
		
\item[\rm (ii)] $\alpha(\mathcal{P}_e(S))$ is the number of maximal monogenic subsemigroup of $S$.
\end{enumerate}	 
\end{theorem}

\begin{proof} (i) Let $x \in S$. Then $x \in \langle y \rangle$ for some $y$, where $ y  \in \mathcal M$. Since the subgraph induced by $\langle y \rangle$ forms a  clique, we get deg$(x) \geq m - 1$. Now choose $z \in S$ such that $ z \in \mathcal M$ and $o(z) = m$. Then deg$(z) = m-1$. Thus, we have the result.
	
	\smallskip
	
	\noindent		
	(ii) First note that for $x, y \in \mathcal M$ such that $\langle x \rangle \ne \langle y \rangle$, we have $x \nsim y$. For instance, if $x \sim y$ then $x, y \in  \langle z \rangle$ for some $z \in S$. Consequently, $\langle x \rangle = \langle y \rangle = \langle z \rangle$; a contradiction. Thus, $\alpha(\mathcal{P}_e(S))$ is more than or equal to number of maximal monogenic subsemigroup of $S$. Further, observed that $S$ is the union of maximal monogenic subsemigroup of $S$ and the subgraph induced by maximal monogenic subsemigroups forms a clique in $\mathcal{P}_e(S)$. It follows that $\alpha(\mathcal{P}_e(S))$  is less than or equal to number of maximal monogenic subsemigroup of $S$. Hence, the result holds .
\end{proof}

\section{Chromatic Number of $\mathcal{P}_e(S)$}

 It is well known that the power graph is a spanning subgraph of cyclic graph and cyclic graph is a spanning subgraph of enhanced power graph of a semigroup $S$. In \cite{a.shitov}, Shitov proved that the chromatic number of power graph of an arbitrary semigroup is at most countable. Then Dalal et al. have shown that the chromatic number of cyclic graph of an arbitrary semigroup is at most countable (see \cite{a.Dalal2020chromatic}). But this result need not hold in the case of enhanced power graph associated of semigroup. In this section, we construct a semigroup $S$ such that the chromatic number of $\mathcal{P}_e(S)$ is uncountable (See Example \ref{chr.-uncountable}).

 First, we recall some basic definition and notations. For any relation $R$ on a set $X$, we define $R^{-1}$ by
\[R^{-1} = \{(x, y) \in X \times X : \; (y, x) \in R \} \]

and $1_X$ denote the identity relation on $X$. Now we define another relation
\[R^{\infty} = \bigcup \{R^n : \; n \geq 1 \}, \]
where $R^n$ is the $n$ time composition of $R$. We denote the relation $R^e$ is the smallest equivalence relation containing $R$.

\begin{proposition}[{\rm \cite[Proposition 1.4.9]{b.Howie}}]
	For every relation $R$ on a set $X$, we have $ R^e = [R \cup R^{-1} \cup 1_X]^{\infty}$.
\end{proposition}

\begin{proposition}[{\rm \cite[Proposition 1.4.10]{b.Howie}}]\label{ch1-R^e}
	If $R$ is a relation on a set $X$ and $R^e$ is the smallest equivalence relation on $X$ containing $R$, then $(x, y) \in R^{e}$ if and only if either $x = y$ or, for some $n \in \mathbb N$, there is a sequence of transitions
	\[x = z_1 \rightarrow z_2 \rightarrow \cdots \rightarrow z_n = y\]
	in which, for each $i$ in $\{1, 2, \ldots, n-1\}$, either $(z_i, z_{i + 1}) \in R$ or $(z_{i + 1}, z_{i}) \in R$.
\end{proposition}

\begin{definition}
	Let $S$ be a semigroup.  A relation $R$ on  $S$ is \emph{left compatible} if 
	\[ (\forall a , s ,t \in S)  \ \ (s,t) \in R  \ \ \Rightarrow (as , at) \in R,\]
	
	\noindent and \emph{right compatible} if 
	\[ (\forall a , s ,t \in S)  \ \ (s,t) \in R  \ \ \Rightarrow (sa , ta) \in R.\]

	\noindent It is called \emph{compatible}\index{compatible} if 
	\[ (\forall s , s' ,t , t' \in S)  \ \ (s,s') \in R \ \ and  \ \  (t,t')  \in R  \ \ \Rightarrow (ss' , tt') \in R.\]
	
	\noindent A left [right] compatible equivalence relation is called \emph{left [right] congruence}. A compatible equivalence relation is called \emph{congruence}\index{congruence}. 
\end{definition}

\begin{proposition}[{\rm \cite[Proposition 1.5.1]{b.Howie}}]
A relation $\rho$ on a semigroup $S$ is a congruence if and only if it is both left and right congruence.
\end{proposition}

\begin{theorem}[{\rm \cite[Theorem 1.5.4]{b.Howie}}]\label{ch1-quotient-semi}
	Let $S$ be semigroup and let $\rho$ be a congruence on $S$. Then $S/ \rho = \{a \rho : \; a \in S \}$ is a semigroup with respect to the operation is defined by
	$(a\rho)(b\rho) = (ab)\rho$.
\end{theorem}

The smallest  congruence relation containing $R$ is denoted by $R^{\#}$. By $S^1$ we shall mean the monoid obtained from $S$ by adjoining an identity element ( if $S$ does not already have such an element). Now we define another relation $R^c$ by
\[R^c = \{ (xay , xby) : x ,y \in S^1 \ \ , (a,b) \in R \}.\]

\begin{lemma}[{\rm \cite[Proposition 1.5.5]{b.Howie}}]
	The relation $R^c$ is the smallest left and right compatible containing $R$.
\end{lemma}

\begin{proposition}[{\rm \cite[Proposition 1.5.8]{b.Howie}}]\label{ch1-rho-sharp}
	For every relation $R$ on a semigroup $S$, we have $R^{\#} = (R^c)^e$.
\end{proposition} 

 In the following example, we construct a semigroup $S$ such that $\chi(\mathcal{P}_e(S))$ is uncountable.

\begin{example}\label{chr.-uncountable}
	Consider the sets $B =  \{(i, j) \in [1, 2] \times [1, 2] : \; j < i \}$ and $A = B \cup [1, 2]$. Now we define a relation $\rho$ on $A$ by
	\[ (i, j)^2 \rho \; i \;  {\rm and} \; (i, j)^3 \rho \; j \; {\rm for \; all} \; (i, j) \in B. \]
	We prove that $\chi(\mathcal{P}_e(A^*/ \rho^{\#}))$ is uncountable. First we claim that: $i \rho^{\#} \ne j \rho^{\#}$ for $i, j \in [1, 2]$. On contrary, we assume that   $i \rho^{\#} = j \rho^{\#}$ for some $i, j \in [1, 2]$. Then $ i \; \rho^{\#} j$. Since $\rho^{\#} = (\rho^c)^e$ (see Proposition \ref{ch1-rho-sharp}) so that there exist $z_1, z_2, \ldots, z_n \in A^*$ such that $z_1 = i$ and $z_n = j$ with $z_k \; \rho^c \; z_{k + 1}$ for all $k$, where $1 \leq k \leq n -1$. Consequently, we get $z_k = x_k a_k y_k$ and $z_{k + 1} = x_k b_k y_k$ for some $x_k, y_k, a_k, b_k \in A^*$ and $a_k \; \rho \; b_k$. Observe that $a_k, b_k \in \{u, (v, w)^2, (v', w')^3\}$ for some $u, v, w, v', w' \in [1, 2]$.  For $i = z_1 = x_1 a_1 y_1$, we have $x_1 = y_1 = \epsilon$ and $a_1 = i$. Then $z_2 = b_1$ because $z_2 = x_1 b_1 y_1$.  For $z_2 = b_1 \; \rho \; i$, we get either $z_2 = (i, t)^2$ or $z_2 = (t, i)^3$, where $t \in [1, 2]$. Suppose $z_2 = (i, t)^2$. Since $z_2 = x_2 a_2 y_2$ and so $x_2 = y_2 = \epsilon$ and $a_2 = (i, t)^2$. On continuing this process we obtain either $z_n = i$ or $z_n = (i, t)^2$ which is not possible because $j \ne i$ and $j \ne (i, t)^2$.  Similarly we get a contradiction when $z_2 = (t, i)^3$. This completes the proof of claim.
	
	Let $i, j \in [1, 2]$. Without loss of generality, we assume that $j < i$. Then $i \rho^{\#} \ne j \rho^{\#}$. Since $\rho^{\#}$ is a congruence  so that $(i, j)^m \rho^{\#} = ((i, j)\rho^{\#})^m$ for all $m \in \mathbb N$. Therefore, we have $i \rho^{\#} = (i, j)^2 \rho^{\#} = ((i, j)\rho^{\#})^2$ and $j \rho^{\#} = (i, j)^3 \rho^{\#} = ((i, j)\rho^{\#})^3$. This implies that $i \rho^{\#}, j \rho^{\#} \in \langle (i, j)\rho^{\#}\rangle$ gives $i \rho^{\#} \sim j \rho^{\#}$ in $\mathcal{P}_e(A^*/\rho^{\#})$. Thus the uncountable set $C = \{i\rho^{\#} : \; i \in [1, 2] \}$ forms a clique in $\mathcal{P}_e(A^*/\rho^{\#})$. By $\omega(\mathcal{P}_e(A^*/\rho^{\#})) \leq \chi(\mathcal{P}_e(A^*/\rho^{\#}))$, the result holds. \hfill $\square$
	\end{example}

\section{Acknowledgement}
The second author wishes to acknowledge the support of MATRICS Grant  (MTR/2018/000779) funded by SERB, India.


\begin{thebibliography}{10}
\bibitem{a.Aalipour2017}
	G.~Aalipour, S.~Akbari, P.~J. Cameron, R.~Nikandish, and F.~Shaveisi.
	\newblock On the structure of the power graph and the enhanced power graph of a
	group.
	\newblock {\em Electron. J. Combin.}, 24(3):$\# P$3.16, 2017.
	
	\bibitem{a.abawajy2013power}
	J.~Abawajy, A.~Kelarev, and M.~Chowdhury.
	\newblock Power graphs: A survey.
	\newblock {\em Electron. J. Graph Theory Appl.}, 1(2):125--147, 2013.
	
	\bibitem{a.Araujo2015}
	J.~Ara\'ujo, W.~Bentz, and J.~Konieczny.
	\newblock The commuting graph of the symmetric inverse semigroup.
	\newblock {\em Israel J. Math.}, 207(1):103--149, 2015.
	
	\bibitem{a.Araujo2011}
	J.~Ara{\'{u}}jo, M.~Kinyon, and J.~Konieczny.
	\newblock Minimal paths in the commuting graphs of semigroups.
	\newblock {\em European J. Combin.}, 32(2):178--197, 2011.
	
	\bibitem{a.Bera2017}
	S.~Bera and A.~K. Bhuniya.
	\newblock On enhanced power graphs of finite groups.
	\newblock {\em J. Algebra Appl.}, 17(8):1850146, 2017.
	
	\bibitem{a.Bera2021EPG}
	S.~Bera, H.~K. Dey, and S.~K. Mukherjee.
	\newblock On the connectivity of enhanced power graphs of finite groups.
	\newblock {\em Graphs Combin.}, 37(2):591--603, 2021.
	
	\bibitem{b.bosak1964graphs}
	J.~Bos\'ak.
	\newblock {\em The graphs of semigroups}.
	\newblock Publ. House Czechoslovak Acad. Sci., Prague, 1964.
	
	\bibitem{a.budden1985cayley}
	F.~Budden.
	\newblock Cayley graphs for some well-known groups.
	\newblock {\em The Mathematical Gazette}, 69(450):271--278, 1985.
	
	\bibitem{a.MKsen2009}
	I.~Chakrabarty, S.~Ghosh, and M.~K. Sen.
	\newblock Undirected power graphs of semigroups.
	\newblock {\em Semigroup Forum}, 78(3):410--426, 2009.
	
	\bibitem{a.Dalal2020chromatic}
	S.~Dalal and J.~Kumar.
	\newblock Chromatic number of the cyclic graph of infinite semigroup.
	\newblock {\em Graphs Combin.}, 36(1):109--113, 2020.
	
	\bibitem{a.dalal2021enhanced}
	S.~Dalal and J.~Kumar.
	\newblock On enhanced power graphs of certain groups.
	\newblock {\em Discrete Mathematics, Algorithms and Applications},
	13(01):2050099, 2021.
	
	\bibitem{a.Dupont2017quotient}
	L.~A. {Dupont}, D.~G. {Mendoza}, and M.~{Rodr{\'{\i}}guez}.
	\newblock {The enhanced quotient graph of the quotient of a finite group}.
	\newblock {\em arXiv:1707.01127}, 2017.
	
	\bibitem{a.Dupont2017}
	L.~A. {Dupont}, D.~G. {Mendoza}, and M.~{Rodr{\'{\i}}guez}.
	\newblock {The rainbow connection number of enhanced power graph}.
	\newblock {\em arXiv:1708.07598}, 2017.
	
	\bibitem{b.Howie}
	J.~M. Howie.
	\newblock {\em Fundamentals of semigroup theory}.
	\newblock Oxford University Press, Oxford, 1995.
	
	\bibitem{a.kelarev2002directed}
	A.~Kelarev and S.~Quinn.
	\newblock Directed graphs and combinatorial properties of semigroups.
	\newblock {\em J. Algebra}, 251(1):16--26, 2002.
	
	\bibitem{a.kelarev2001powermatrices}
	A.~Kelarev, S.~Quinn, and R.~Smolikova.
	\newblock Power graphs and semigroups of matrices.
	\newblock {\em Bull. Austral. Math. Soc.}, 63(2):341--344, 2001.
	
	\bibitem{a.kelarev2009mining}
	A.~Kelarev, J.~Ryan, and J.~Yearwood.
	\newblock Cayley graphs as classifiers for data mining: the influence of
	asymmetries.
	\newblock {\em Discrete Math.}, 309(17):5360--5369, 2009.
	
	\bibitem{a.Kelerve-minimal-automata}
	A.~V. Kelarev.
	\newblock Labelled {C}ayley graphs and minimal automata.
	\newblock {\em Australas. J. Combin.}, 30:95--101, 2004.
	
	\bibitem{a.konieczny2002semigroups}
	J.~Konieczny.
	\newblock Semigroups of transformations commuting with idempotents, 2002.
	
	\bibitem{2019Mametric}
	X.~Ma and Y.~She.
	\newblock The metric dimension of the enhanced power graph of a finite group.
	\newblock {\em J. Algebra Appl.}, 19(01):2050020, 2020.
	
	\bibitem{a.Panda-enhanced}
	R.~P. Panda, S.~Dalal, and J.~Kumar.
	\newblock On the enhanced power graph of a finite group.
	\newblock {\em Comm. Algebra}, 49(4):1697--1716, 2021.
	
	\bibitem{a.shitov}
	Y.~Shitov.
	\newblock Coloring the power graph of a semigroup.
	\newblock {\em Graphs Combin.}, 33(2):485--487, 2017.
	
	\bibitem{a.trotter1978cartesian}
	W.~T. Trotter, Jr. and P.~Erd\H{os}.
	\newblock When the {C}artesian product of directed cycles is {H}amiltonian.
	\newblock {\em J. Graph Theory}, 2(2):137--142, 1978.
	
	\bibitem{b.West}
	D.~B. West.
	\newblock {\em Introduction to Graph Theory}.
	\newblock Second edition, Prentice Hall, 1996.
	
	\bibitem{a.witte1984survey}
	D.~Witte and J.~A. Gallian.
	\newblock A survey: Hamiltonian cycles in {C}ayley graphs.
	\newblock {\em Discrete Math.}, 51(3):293--304, 1984.
	
	\bibitem{a.2019study}
	S.~Zahirovi\'{c}, I.~Bo\v{s}njak, and R.~Madar\'{a}sz.
	\newblock A study of enhanced power graphs of finite groups.
	\newblock {\em J. Algebra Appl.}, 19(4):2050062, 2020.
	
\end{thebibliography}
\end{document}